\pdfoutput=1

\documentclass[options]{article}
\usepackage{amsmath}
\usepackage{amsfonts}
\usepackage{amssymb}
\usepackage{amsthm}
\usepackage{mathtools}
\usepackage{titling}
\usepackage{authblk}
\usepackage{commath}
\usepackage{xcolor}
\usepackage{graphicx}
\usepackage{subcaption}
\usepackage{bbm}
\usepackage{hyperref}
\usepackage{color,soul}
\usepackage{bm}
\usepackage{mathrsfs}

\usepackage{enumitem}
\newtheorem{remark}{Remark}[section]
\newtheorem{theorem}[remark]{Theorem}
\newtheorem{corollary}[remark]{Corollary}
\newtheorem{lemma}[remark]{Lemma}
\usepackage[margin=1in]{geometry}

\newtheorem{proposition}[remark]{Proposition}

\numberwithin{equation}{section}
\allowdisplaybreaks

\usepackage[backend=bibtex,style=trad-plain,giveninits=true,maxnames=10,isbn=false,url=false]{biblatex}
\newcommand{\nocontentsline}[3]{}
\newcommand{\tocless}[2]{\bgroup\let\addcontentsline=\nocontentsline#1{#2}\egroup}
\usepackage{tocloft}
\usepackage{blindtext}
 \cftpagenumbersoff{section}
 \cftpagenumbersoff{subsection}
 \cftpagenumbersoff{subsubsection}
\makeatletter \renewcommand{\@cftmaketoctitle}{} \makeatother
\usepackage{bookmark}

\title{On the Stability of the $s$-Nonlocal $p$-Obstacle Problem and their Coincidence Sets and Free Boundaries}
\author{Catharine W. K. Lo\thanks{Department of Mathematics, City University of Hong Kong, Kowloon, Hong Kong\\ Email address: wingkclo@cityu.edu.hk} \, and Jos\'e Francisco Rodrigues\thanks{CMAFcIO -- Departamento de Matem\'atica, Faculdade de Ci\^encias, Universidade de Lisboa P-1749-016 Lisboa, Portugal\\ Email address: jfrodrigues@ciencias.ulisboa.pt}}
\date{}

\begin{document}
\maketitle

\begin{abstract}
    We show that the solutions to the nonlocal obstacle problems for the nonlocal $-\Delta_p^s$ operator, when the fractional parameter $s\to\sigma$ for $0<\sigma\leq1$, converge to the solution of the corresponding obstacle problem for $-\Delta_p^\sigma$, being $\sigma=1$ the classical obstacle problem for the local $p$-Laplacian. We discuss the weak stability of the quasi-characteristic functions of coincidence sets of the solution with the obstacle, which is a strong convergence of their characteristic functions when $s\nearrow 1$ under a nondegeneracy condition.  This stability can be shown also in terms of the convergence of the free boundaries, as well as of the coincidence sets, in Hausdorff distance when $s\nearrow 1$, under non-degeneracy local assumptions on the external force and a local topological property of the coincidence set of the limit classical obstacle problem for the local $p$-Laplacian, essentially when the limit coincidence set is the closure of its interior. 
\end{abstract}

\tocless\section{Introduction}
\hypertarget{introduction}{}
\bookmark[level=section,dest=introduction]{Introduction}

The $s$-nonlocal $p$-obstacle problem,  in a Lipschitz bounded domain $\Omega\subset \mathbb{R}^n$, consists of finding
\begin{equation}\label{ObsProb}
u\in \mathbb{K}^s:\quad
    \langle-\Delta_p^su-f,v-u\rangle\geq0\quad\forall v\in \mathbb{K}^s ,
\end{equation} 
for $f\in L^\infty(\Omega)$, 
in the closed convex set of the fractional Sobolev space $W^{s,p}_0(\Omega)$, for $0<s<1$ and $1<p<\infty$, given by 
\begin{equation}\label{ConvexSetDef}\mathbb{K}^s=\{v\in W^{s,p}_0(\Omega):v\geq0 \text{ a.e. in }\Omega \}.\end{equation} Here, $\langle\cdot,\cdot\rangle$ denotes the duality bracket between $W^{s,p}_0(\Omega)$ and its dual $W^{-s,p'}(\Omega)\supset L^\infty(\Omega)$, for $p'=p/(p-1)$, and the nonlocal nonlinear homogeneous fractional $p$-Laplacian $\Delta_p^s$ is defined by 
\begin{equation}\label{LpLap} \langle-\Delta_p^su, v\rangle= \kappa_s\int_{\mathbb{R}^n}\int_{\mathbb{R}^n}\frac{|u(x)-u(y)|^{p-2}(u(x)-u(y))(v(x)-v(y))}{|x-y|^{n+sp}}\,dx\,dy \quad u,v\in W^{s,p}_0(\Omega),\end{equation} with $u,v=0$ a.e. in  $\mathbb{R}^n\backslash\Omega$ and the normalising constant $\kappa_s=(1-s)C_{n,p}/2 $ has the constant $C_{n,p}$ given by $C_{n,p}^{-1}=\frac{1}{p}\int_{\mathbb{S}^{n-1}}\omega^p_n\,dS_\omega$ for the unit sphere $\mathbb{S}^{n-1}$ in $\mathbb{R}^n$, as in \cite{BonderSalort2020StabilityNonlocal}. 
Note that the constant $\kappa_s$ vanishes as $s\nearrow1$ and is such that the fractional $p$-Laplacian $\Delta_p^s\varphi$ tends, in the distributional sense for smooth functions $\varphi$, to the classical $p$-Laplacian $\Delta_p\varphi=\nabla\cdot(|\nabla\varphi|^{p-2}\nabla\varphi)$ in the limit case $s=1$. The vanishing term $1-s$ is absent in several other definitions in the literature (see, for instance,  \cite{LindgrenLindqvist2014CVPDEFractionalEigenvalues}, \cite{brasco2015stability}, \cite{IannizzottoMosconiSquassina2016RMI_FracPLapGlobalHolder} and \cite{IannizzottoMosconiSquassina2020JFA_FracPLapTMonotoneLewyStampacchia}). In the linear case where $p=2$, we have the fractional Laplacian in \eqref{LpLap} (see for instance \cite{FracObsRiesz}, \cite{RosOton2016Survey}), where it is known that the obstacle problem transforms into a degenerate classical Signorini problem in the cylinder $\Omega\times(0,\infty)$, by performing the Caffarelli-Silvestre extension to $\mathbb{R}^{n+1}$ (see, for instance, \cite{CaffarelliSalsaSilvestre}).

After recalling, in Section 2, some useful properties on the fractional solutions to the nonlocal semilinear Dirichlet problem, including Hölder regularity estimates, we use, in Section 3, the bounded penalization to the obstacle problem to obtain global and uniform local Hölder continuity to the solutions of the $s$-nonlocal $p$-obstacle problem \eqref{ObsProb}.

The purpose of this work is to extend to those solutions $u^s$ of \eqref{ObsProb} the continuous dependence with respect to the fractional parameter $s$ in suitable fractional Sobolev spaces. This is done in Section 4, including the case $s\nearrow1$, which limit corresponds to the classical $p$-obstacle problem \eqref{ObsProb}  with
$u^1\in \mathbb{K}^1$ in the usual Sobolev space $W^{1,p}_0(\Omega)$, $1<p<\infty$. The Hilbertian case $p=2$, corresponding to the fractional Laplacian, has been considered in Remark 3.9 of \cite{FracObsRiesz}. Stability results for the eigenfunctions for the nonlocal $p$-Laplacian in fractional norms have been obtained earlier in \cite{brasco2015stability} and extended to solutions of certain semilinear Dirichlet problems in \cite{BonderSalort2020StabilityNonlocal}
and in \cite{BonderSalort2022AsymptoticSFracEnergies}.

Our main new results concern the local stability, under suitable nondegeneracy conditions on $f$ and on the free boundary $\Phi^1=\partial\{u^1>0\}\cap\Omega$  of the limit problem, of the coincidence sets $I^s=\{u^s=0\}$, and also of $\Phi^s$, with respect to $s\nearrow1$.

First, in Section 5, we observe that the solution $u^s$ to \eqref{ObsProb}, with $0<s<1$, also solves a semilinear Dirichlet problem ($f^-=f^+-f, f^+=f\vee 0$)
\begin{equation}-\Delta_p^su^s=f-f^-(1- \vartheta^s(u^s))\quad\text{ a.e. in }\Omega,\quad\quad u^s=0\quad\text{ a.e. in }\mathbb{R}^n\backslash\Omega,\end{equation}
involving a quasi-characteristic function $\vartheta^s=\vartheta^s(u^s)$, such that
\begin{equation}\label{1-ThetaLimit}
    0 \leq 1-\vartheta^s \leq \chi_{\{u^s=0\}}\leq 1 \quad \text{ a.e. in }\Omega.
\end{equation}
where $\chi_{\{u^s=0\}}$ denotes the characteristic function of the coincidence set $I^s$ ($\chi_{\{u^s=0\}}=1$ if $x\in I^s$ and  $\chi_{\{u^s=0\}}=0$ otherwise). Then, we show that the assumption $f\neq 0$ a.e. in $\omega$, for any open subset of $\Omega$, implies $\vartheta^s \rightharpoonup\vartheta^\sigma$ in $L^\infty(\omega)$-weak$^*$  as $s\to\sigma$, for $0<\sigma\leq 1$.  Note that if $f\leq 0$ a.e. in $\Omega$ then the solution of the obstacle problem is trivial, i.e.  $u^s=0$ in $\Omega$. Using the additional property on the local vanishing of the Lebesgue measure of $\Phi^1\cap\omega$ in the limit problem, which holds if $f(x) \leq -\lambda <0, \text{ a.e. }x \in \omega\Subset\Omega$, we prove the convergence $\chi_{\{u^s=0\}} \to \chi_{\{u^1=0\}}$ in $L^r(\omega), \forall 1\leq r<\infty$, as $s\nearrow1$, since the limit solution $u^1$ satisfies
\begin{equation}-\Delta_pu^1=f-f^-\chi_{\{u^1=0\}}\quad\text{ a.e. in }\omega\Subset\Omega.\end{equation}

Finally, in Section 6, under a topological property of the coincidence set $I^1$ of the limit $p$-obstacle problem, essentially imposing that it coincides with the closure of its interior, similarly to the classical framework of Chapter 6 of \cite{ObstacleProblems}, we prove the local convergence in $\omega\Subset\Omega$ of $I^s\to I^1$ as $s\nearrow 1$ in the Hausdorff distance, under the assumption $f\neq 0$ a.e. in $\bar{\omega}\subset\Omega$. In addition, if this later nondegeneracy assumption is replaced by $f\leq -\lambda <0$ in $\omega\Subset\Omega$, using also the uniform convergence $u^s\to u^1$ in $\bar{\omega}$ we prove the local convergence of the free boundaries $\Phi^s\cap\bar{\omega}\to \Phi^1\cap\bar{\omega}$ still in the Hausdorff distance.

Although the techniques used here, by imposing nondegeneracy conditions on the coincidence set only in the limit problem, are similar to those of the classical local obstacle problem in \cite{ObstacleProblems}
and \cite{JFR2005j1}, these stability results of the coincidence sets and their free boundaries in the nonlocal framework, where their regularity is unknown, are completely new.

.

\section{Preliminaries on the Semilinear Dirichlet Problem}

We use the definition of the fractional Sobolev-Gagliardo spaces $W^{s,p}_0(\Omega)$ as the closure of  $C_c^\infty(\Omega)$ in
\begin{equation}\label{NormDef}
    W^{s,p}(\mathbb{R}^n)=\left\{u\in L^p(\mathbb{R}^n) :  [u]_{s,p}=\left((1-s)C_{n,p}\int_{\mathbb{R}^n}\int_{\mathbb{R}^n}\frac{|u(x)-u(y)|^p}{|x-y|^{n+sp}}dxdy\right)^\frac{1}{p}< \infty  \right\}.
\end{equation} 
We observe that, except for $s=1/p$, we have $W^{s,p}_0(\Omega)=\{u\in W^{s,p}(\mathbb{R}^n): u=0 \text{ a.e. in }\mathbb{R}^n\backslash\Omega\}$ when $\Omega$ is a bounded Lipschitz domain (see for instance \cite{EdmundsEvansBook}). 
When $p=2$, we have the fractional Hilbertian Sobolev space $H^s_0(\Omega)=W^{s,2}_0(\Omega)$, which has dual $H^{-s}(\Omega)=W^{-s,2}(\Omega)$.


We recall the Poincar\'e inequality for fractional Sobolev spaces $W^{s,p}_0(\Omega)$, which implies that the Gagliardo seminorm $[u]_{s,p}$, defined in \eqref{NormDef}, is a norm in $W^{s,p}_0(\Omega)$.
\begin{lemma}[Poincar\'e inequality, Theorem 6.1 of \cite{BonderSalort2019JFAFracOrliczSobolev}]\label{Poincare}
Let $0<s_0<s<1$ and $1<p<\infty$. For any open bounded set $\Omega\subset\mathbb{R}^n$, there exists a Poincar\'e constant $c_P>0$, depending only on $\Omega, s_0$ and $n$, such that \[\norm{u}_{L^p(\Omega)}^p\leq c_P[u]_{s,p}^p\] for every $s$, $s_0\leq s<1$, and for all $u\in W^{s,p}_0(\Omega)$. 
\end{lemma}

We also use the fractional Sobolev embedding type theorems in the following form (see for instance \cite{HitchhikerGuide}).
\begin{lemma}\label{Sobolev} Let $0<s<1$. For any open bounded set $\Omega\subset\mathbb{R}^n$, for $u\in W^{s,p}_0(\Omega)$, there exists a Sobolev-Poincar\'e constant $C_{SP}>0$ such that 
    \[\norm{u}_{L^{p^*_s}(\Omega)}\leq C_{SP}[u]_{s,p},\] where $p^*_s>p$ is the fractional Sobolev exponent such that $\frac{1}{p^*_s}=\frac{1}{p}-\frac{s}{n}$ if $sp<n$, any $p^*_s<\infty$ if $sp=n$ and $p^*_s=\infty$ if $sp>n$. Moreover the embedding $W^{s,p}_0(\Omega)\subset L^q(\Omega)$ is also compact for $1\leq q<p^*_s$.
\end{lemma}

\begin{remark}
    For $v\in W^{s,p}_0(\Omega)\subset L^{p^*_s}(\Omega)$, if $u\in W^{s,p}_0(\Omega)$ is such that $-\Delta_p^su\in L^{p^\natural_s}(\Omega)\subset W^{-s,p'}(\Omega)$, one can write 
    \[\langle -\Delta_p^su,v\rangle =  \int_\Omega(-\Delta_p^su)v\, dx,\]
    where $p^\natural_s=(p^*_s)'=p^*_s/(p^*_s-1)$ is the conjugate fractional Sobolev exponent, such that $\frac{1}{p^\natural_s}=1-\frac{1}{p}+\frac{s}{n}$ if $sp<n$, any $p^\natural_s>1$ if $sp=n$ and $p^\natural_s\geq 1$ if if $sp>n$.
\end{remark}

In order to recall some of the main properties of the fractional $p$-Laplacian, in particular, estimates and comparison properties for the solutions of the Dirichlet problem,  we use the standard notation for the positive and negative parts of $v$, $v^+\equiv v\vee0$ and $v^-\equiv-v\vee0=-(v\wedge0)$, and we recall the Jordan decomposition of $v$ given by $v=v^+-v^-$ and $|v|\equiv v\vee(-v)=v^++v^-$. The use of classical inequalities with sharp constants, used in the local $p$-Laplacian in \cite{SimonPLap}, can also give interesting results in the nonlocal range $0<s<1$.

\begin{proposition}\label{LgCoercive}
The operator $-\Delta_p^s$ is strongly coercive, in the sense that 
\begin{equation}\label{LgCoerciveEq}\langle -\Delta_p^su- (-\Delta_p^s)v,u-v\rangle
\geq 
\begin{cases} 2^{1-p}[u-v]_{s,p}^p&\text{ if }p\geq2,\\
(p-1)2^{\frac{p^2-4p+2}{p}}\dfrac{[u-v]_{s,p}^2}{\left([u]_{s,p}+[v]_{s,p}\right)^{2-p}}&\text{ if }1<p<2.\end{cases}\end{equation}
Moreover it is also strictly T-monotone in $W^{s,p}_0(\Omega)$, i.e. \begin{equation}\label{LgTMonotoneEq}
    \langle-\Delta_p^su-(-\Delta_p^s)v,(u-v)^+\rangle>0\quad\forall u,v\in W^{s,p}_0(\Omega):(u-v)^+\neq 0.\end{equation}

\end{proposition}

\begin{proof} For T-monotonicity see, for instance, Remark 3.3 and pages 261--262 of \cite{GMosconi} or Equation (2.13) of \cite{IannizzottoMosconiSquassina2020JFA_FracPLapTMonotoneLewyStampacchia}.

For the strong coercivity, we have 
    \begin{align*}&\,\langle -\Delta_p^su- (-\Delta_p^s)v,u-v\rangle
\\&=\kappa_s\int_{\mathbb{R}^d}\int_{\mathbb{R}^d}\frac{\left[|u(x)-u(y)|^{p-2}(u(x)-u(y)) - |v(x)-v(y)|^{p-2}(v(x)-v(y)) \right] }{|x-y|^{n+sp}} 
\\ &\quad\quad\quad\quad\quad\quad \times \left[((u(x)-v(x))-(u(y)-v(y)))\right]\,dx\,dy\\&\geq 
\begin{cases} 2^{1-p}[u-v]_{s,p}^p&\text{ if }p\geq2,\\
(p-1)2^{\frac{p^2-4p+2}{p}}\dfrac{[u-v]_{s,p}^2}{\left([u]_{s,p}+[v]_{s,p}\right)^{2-p}}&\text{ if }1<p<2
\end{cases}\end{align*} by the definition of the Gagliardo seminorm $[u]_{s,p}$, where $(1-s)C_{n,p}=2\kappa_s$, and the well-known inequality (see Lemma 1.11 of \cite{EdmundsEvansBook}), with $a=u(x)-u(y)$ and $b=v(x)-v(y)$
\begin{equation}\label{pLapIneq}
(|a|^{p-2}a-|b|^{p-2}b)(a-b)\geq 
\begin{cases} 2^{2-p}|a-b|^p&\text{ if }p\geq2,\\
(p-1)\frac{|a-b|^2}{(|a|+|b|)^{2-p}}&\text{ if }1<p<2
\end{cases} 
\quad\forall a,b\in\mathbb{R}.\end{equation} In the second case $1<p<2$, we have also used the well-known inequality 
\[(a+b)^p \leq 2^{(p-1)^+}(a^p+b^p)\quad \forall a,b,p>0\] 
and the reverse H\"older inequality with $0<p/2<1$ and $p/(p-2)<0$, in the condensated form
\begin{multline*}
    (1-s) C_{n,p}\int_{\mathbb{R}^d}\int_{\mathbb{R}^d}\frac{|a-b|^2}{(|a|+|b|)^{2-p}}\,\frac{dx\,dy}{|x-y|^{n+sp}} \\\geq \left[(1-s) C_{n,p}\int_{\mathbb{R}^d}\int_{\mathbb{R}^d}|a-b|^p\,\frac{dx\,dy}{|x-y|^{n+sp}}\right]^{2/p}\left[(1-s) C_{n,p}\int_{\mathbb{R}^d}\int_{\mathbb{R}^d}(|a|+|b|)^{p}\,\frac{dx\,dy}{|x-y|^{n+sp}}\right]^{1-2/p}.
\end{multline*}

\end{proof}

With these properties and a monotonicity assumption, we can complement and improve an existence result on the semilinear Dirichlet problem for the nonlocal $p$-Laplacian. Let $F:\Omega\times\mathbb{R}\to\mathbb{R}$ be a Carath\'eodory function, monotone non-increasing in $z$, satisfying the growth condition 
\begin{equation}\label{FGrowth}|F(x,z)|\leq f(x)+C|z|^{p^\natural_s}\end{equation} for some constant $C\geq0$ and a function $f\in L^{p^\natural_s}(\Omega)$ with $p^\natural_s$ being the conjugate $s$-fractional Sobolev exponent, given in Remark 2.3. 
For $0<s<1$, we consider solutions  $u^s\in W^{s,p}_0(\Omega)$  of the semilinear Dirichlet problem
\begin{equation}\label{DirichletExistEq}
-\Delta_p^su=F(x,u)\quad\text{ in }\Omega,\quad\quad u=0\quad\text{ in }\mathbb{R}^n\backslash\Omega.
\end{equation}

\begin{theorem}\label{DirichletExistThm}

Let $F$ be a monotonically non-increasing in $z$ Carathéodory function satisfying \eqref{FGrowth}. Then there exists a unique solution to the semilinear Dirichlet problem \eqref{DirichletExistEq}. Let  $u_1,u_2$ be solutions corresponding to $F_1,F_2$. Then the following estimate holds
\[[u_1-u_2]_{s,p}\leq  \begin{cases} 2C_{SP}^\frac{1}{p-1}\norm{F_1(x,0)-F_2(x,0)}^\frac{1}{p-1}_{L^{p^\natural_s}(\Omega)}&\text{ if }p\geq2,\\
C_p(F_1,F_2)\norm{F_1(x,0)-F_2(x,0)}_{L^{p^\natural_s}(\Omega)}&\text{ if }1<p<2,\end{cases}\] where $C_p(F_1,F_2)=(p-1)^{-1}2^{\frac{2}{p}}C_{SP}^\frac{3-2p}{1-p}\left(\norm{F_1(x,0)}^\frac{1}{1-p}_{L^{p^\natural_s}(\Omega)}+\norm{F_2(x,0)}^\frac{1}{1-p}_{L^{p^\natural_s}(\Omega)}\right)^{2-p}$, i.e. the solution map   $F\mapsto u$ is $\frac{1}{p-1}$-H\"older continuous when $p\geq2$, and locally Lipschitz continuous when $1<p<2$.
\end{theorem}

\begin{proof}
    The existence and uniqueness follow immediately by the general theory, since $(-\Delta_p^s-F)(\cdot)$  is easily seen to be a strictly monotone operator in $W^{s,p}_0(\Omega)$. Writing, for $i=1,2$, 
    \begin{equation}\label{DirichletVarEq}\int_\Omega(-\Delta_p^su_i - F_i(x,u_i))v\, dx=0\end{equation}
    the strong continuity estimates follow easily by using the test function $v=u_2-u_1$ by applying the inequalities \eqref{pLapIneq} and the monotonicity of $F_i$. For the case of $1<p<2$, we complete the proof by applying the H\"older inequality and the monotonicity of $F$ in order to obtain, by letting $v=u_i$ in \eqref{DirichletVarEq}
    \[[u_i]_{s,p}\leq \left(2C_{SP}\norm{F_i(x,0)}_{L^{p^\natural_s}(\Omega)}\right)^\frac{1}{p-1}.\]
\end{proof}



    

We define the functionals $\mathcal{J}_{s,p}:W^{s,p}_0(\Omega)\to \mathbb{R}$ for $0<s\leq1$, as in \cite{BonderSalort2020StabilityNonlocal}, with $W^{1,p}_0(\Omega)$ being the usual Sobolev space, by 
\begin{equation}\label{JFuncDef}
\mathcal{J}_{s,p}(u):=
\begin{cases}
    \frac{1}{p}[u]_{s,p}^p & \quad \text{ for }0<s<1,\\
    \frac{1}{p} \norm{\nabla u}_{L^p(\Omega)}^p & \quad \text{ for }s=1.
\end{cases}
\end{equation} Then, in terms of these functionals, when $s\nearrow1$ the result of \cite{BourgainBrezisMironescu}, $[u]_{s,p}^p \to \norm{\nabla u}_{L^p(\Omega)}^p$ for $u\in W^{1,p}_0(\Omega)\subset W^{s,p}_0(\Omega)$, can be interpreted in the sense of $\Gamma$-convergence as in \cite{BonderSalort2019JFAFracOrliczSobolev} and \cite{BonderSalort2020StabilityNonlocal}. This means that, extending the functionals $\mathcal{J}_{s,p}$ by $+\infty$ to $L^p(\Omega)\backslash W^{s,p}_0(\Omega)$, for any sequence $0<s_k\nearrow1$, the sequence $\{\mathcal{J}_{s_k,p}\}_{k\in\mathbb{N}}$ $\Gamma$-converges to $\mathcal{J}_{1,p}$ in $L^p(\Omega)$.

The Fr\'echet derivative $\mathcal{J}_{s,p}':W^{s,p}_0(\Omega) \to W^{-s,p'}(\Omega)$ of $\mathcal{J}_{s,p}$ is given by 
\begin{equation}
    \langle \mathcal{J}_{s,p}'(u),v\rangle = 
    \begin{dcases}
        \frac{(1-s)C_{n,p}}{2} \iint_{\mathbb{R}^n\times\mathbb{R}^n}\frac{|u(x)-u(y)|^{p-2}(u(x)-u(y))(v(x)-v(y))}{|x-y|^{n+sp}}\,dx\,dy &\text{ for }0<s<1,\\
        \int_\Omega|\nabla u|^{p-2}\nabla u \cdot \nabla v \, dx  &\text{ for }s=1.
    \end{dcases}
\end{equation}
In terms of the fractional Laplacians and of their dualities in $W^{s,p}_0(\Omega)$, for $0<s<1$ and $s=1$, respectively, we observe that $\mathcal{J}_{s,p}'$ may be given by 
\begin{equation}
    \langle \mathcal{J}_{s,p}'(u),v\rangle = 
    \begin{dcases}
        \langle-\Delta_p^su, v\rangle &\text{ for any }u,v\in W^{s,p}_0(\Omega)\text{ for }0<s<1,\\
        \langle-\Delta_pu, v\rangle &\text{ for any }u,v\in W^{1,p}_0(\Omega)\text{ for }s=1
    \end{dcases}
\end{equation}

 We shall use and extend later the following uniform asymptotic expansion of Lemma 2.7 of \cite{BonderSalort2020StabilityNonlocal}:
\begin{lemma}\label{BonderSalortConvgLemma}
    For $u\in W^{1,p}_0(\Omega)$ fixed, let $v_s\in W^{s,p}_0(\Omega)$ for any $0<s\leq1$. Then, \[\mathcal{J}_{s,p}(u+\tau v_s)-\mathcal{J}_{s,p}(u)=\tau \langle -\Delta_p^su,v_s\rangle+o(\tau ) \quad \text{ for all }\tau >0,\] where $o(\tau )$ depends only on $C\geq [v_s]_{s,p}$.
\end{lemma}

This Lemma was used in \cite{BonderSalort2020StabilityNonlocal} to show the stability of solutions $u^s\to u$ in $L^p(\Omega)$ to semilinear Dirichlet problems as $s\to 1$. Observing that $f\in L^{p'}(\Omega)\subset L^{p^\natural_s}(\Omega)$ for all $0<s<1$, in the monotone case we obtain the following stronger result.

\begin{theorem}\label{SemiLinearConvgThm}
    Suppose that the Carathéodory function $F$ is non-increasing in $z$ and satisfies \eqref{FGrowth} with $f\in L^{p'}(\Omega)$. For $0<s_0<s<1$, if $s\to1$ then $u^s\to u$ in $W^{r,p}_0(\Omega)$, for any $0\leq r<1$, where $u\in W^{1,p}_0(\Omega)$ is the solution of 
\begin{equation}\label{DirichletProbEqConvg}
-\Delta_p u=F(x,u)\quad\text{ in }\Omega,\quad\quad u=0\quad\text{ on }\partial\Omega.
\end{equation} 

\end{theorem}

\begin{proof}
    Denoting $u^s$ the solution of \eqref{DirichletExistEq} and letting $v=u^s$ in the corresponding variational formulation \eqref{DirichletVarEq} one obtains
    \begin{equation}\label{DirichletConvgEq}[u^s]_{s,p}\leq \left(2c_{P}\norm{F(\cdot,0)}_{L^{p'}(\Omega)}\right)^\frac{1}{p-1}=\left(2c_{P}\norm{f}_{L^{p'}(\Omega)}\right)^\frac{1}{p-1}.\end{equation}

    Therefore, by the assumption on $F$, there exists a constant $C_1>0$, depending on $f$ and on $\Omega$ but independent of $s>s_0$, such that $\sup_{s} [u^s]_{s,p}\leq C_1$. Then, by  Theorem 3.3 of \cite{BonderSalort2020StabilityNonlocal}, any accumulation point $u$  of any subsequence $\{u^{s_n}\}_{n}$ as $s_n\to1$, in the $L^p(\Omega)$-topology, satisfies $u\in W^{1,p}_0(\Omega)$ and solves weakly the equation \eqref{DirichletProbEqConvg}. Since the subsequence $\{u^{s_n}\}_{n}$ is bounded in $W^{s,p}_0(\Omega)$ for each $s>s_0$, we have the convergence of any of those subsequences also in $W^{s,p}_0(\Omega)$-weak. By the compactness of $W^{s,p}_0(\Omega)\Subset W^{r,p}_0(\Omega)$, for any $0\leq r<s\leq 1$ we finally conclude the strong convergence of the sequence $s\to 1$ in this later space by the uniqueness of the solution of \eqref{DirichletProbEqConvg}.
\end{proof}

\begin{remark}
    When $F$ does not depend on $u$, this is Theorem 6.6 of \cite{BonderSalort2019JFAFracOrliczSobolev}. When $F$ does depend on $u$ without the monotonicity assumption and has a growth \eqref{FGrowth} with the power $p_s^\natural$ replaced by any $q$, such that $1\leq q<p_s^*-1$, this is Theorem 3.3 of \cite{BonderSalort2020StabilityNonlocal}. However the lack of uniqueness in that semilinear Dirichlet problem, in general, implies that the conclusion is only for any accumulation point $u$ of subsequences $u^{s_k}$ for $0<s_k\to 1$.
\end{remark}

Moreover we can use the H\"older regularity estimates for the linear Dirichlet problem in Theorem 1.1 of \cite{IannizzottoMosconiSquassina2016RMI_FracPLapGlobalHolder}, Theorem 8 of \cite{Palatucci2018NonAFracPLapObstacleHolder}, Theorem 6 of \cite{K+Kuusi+Palatucci2016CVPDE-ObstaclePbFracPLap}, or Theorem 1.3 of \cite{Piccinini2022NonA-ObstaclePbFracPLap} to obtain the H\"older continuity of the solution to the semilinear Dirichlet problem up to the boundary, which is exactly the last statement of the following theorem. In addition, its first part concerns the robust local estimates, i.e., estimates that are independent of $s$ far from $0$, which corresponds to Theorem 1.1(i) of \cite{ChakerKimWeidner2022CVPDERegGLap}. Here $C^\alpha(\omega)$ denotes the space of H\"older continuous functions in $\omega$ for some $0<\alpha<1$.     

\begin{theorem}\label{SemiLinearRegThm}
Let $u^s$ denote the solution of the semilinear problem \eqref{ObsProb} with $F$ satisfying \eqref{FGrowth} with $f\in L^\infty(\Omega)$ and $C=0$. 
 For a fixed $s_0$, $0<s_0<1$, and $s$ such that $s_0\leq s<1$, we have $u^s\in C^{\beta}_{loc}(\Omega)$ uniformly in $s$, i.e.  $u^s$ is locally H\"older continuous in $\Omega$, for some $0<\beta\leq s_0$ depending only on $n$ and $p$. This means that for every open subset $\omega\Subset\Omega$ there exists a constant $C_\omega>0$ depending only on $n$, $p$ and $\norm{f}_{L^\infty(\Omega)}$ and independent of $s$, such that $u^s$ satisfies
    \begin{equation}\label{Robust}
    u^s\in C^\beta(\bar{\omega})\quad\text{ with }\quad  \norm{u^s}_{C^\beta(\bar{\omega})}\leq C_\omega \quad \text{for all } s, s_0\leq s<1.\end{equation}

    Moreover, if in addition  $\partial\Omega\in C^{0,1}$, then, $u^s\in W^{s,p}_0(\Omega)\cap C^\alpha(\bar{\Omega})$, i.e.  $u^s$ is globally H\"older continuous in $\Omega$, for some $0<\alpha<1$ depending on $n$, $p$ and $s$, $0<s<1$.

\end{theorem}

As a consequence of Ascoli-Arzelà theorem and the previous results we have the following interesting new result.
\begin{corollary}
    Under assumptions of Theorems \ref{SemiLinearConvgThm} and \ref{SemiLinearRegThm}, there is a $\lambda$, $0<\lambda<\beta$, such that 
    \begin{equation}u^s\to u\quad\text{ in }C^\lambda(\bar{\omega})\quad\text{ as } s\nearrow1, \quad \text{ for every open subset $\omega\Subset\Omega$,}\end{equation} and therefore the convergence is uniform in the interior of $\Omega$.
\end{corollary}

\section{Uniform H\"older Continuity in the Nonlocal $p$-Obstacle Problem}

The nonlocal $p$-obstacle problem \eqref{ObsProb} may be approximated by a family of semilinear nonlocal elliptic Dirichlet problems, with a precise error estimate, by means of a bounded penalisation. As a consequence, this yields uniform estimates for the solutions, which are important to study the stability of the obstacle problem with respect to $s$ for $0<s_0\leq s\leq 1$. We make use of the Stampacchia's method of bounded penalisation, as described, for instance, in Theorem 4.6 of \cite{FracObsRiesz} for the classical obstacle problem with $p=2$ and $s=1$. 

Consider the penalised problem with $f=f^+ -f^-\in L^{p^\natural_s}(\Omega)$ given by 
\begin{equation}\label{PenalProb}u_\varepsilon\in W^{s,p}_0(\Omega):\quad\langle -\Delta_p^s u_\varepsilon,v\rangle+\int_\Omega f^- \theta_\varepsilon (u_\varepsilon) v=\int_\Omega  f^+ v, \quad\forall v\in W^{s,p}_0(\Omega),\end{equation} 
where 
$\theta_\varepsilon(t)$ is an approximation to the multi-valued Heaviside graph defined by
\[\theta_\varepsilon(t)=\theta\left(\frac{t}{\varepsilon}\right),\quad t\in\mathbb{R}\] 
for an arbitrary nondecreasing Lipschitz function $\theta:\mathbb{R}\to[0,1]$ satisfying \begin{equation}\label{ThetaDef}\theta\in C^{0,1}(\mathbb{R}),\quad\theta'\geq0,\quad\theta(+\infty)=1\quad\text{ and }\theta(t)=0\text{ for }t\leq0;\end{equation}\begin{equation}\label{ThetaDef2}\exists C_\theta>0:[1-\theta(t)]t\leq C_\theta,\quad t>0.\end{equation}  

By Theorem \ref{DirichletExistThm}, there exists a unique solution $u_{\varepsilon}$ to the monotone perturbation of the nonlocal $p$-Laplacian Dirichlet problem \eqref{PenalProb}, which converges to the solution of the obstacle problem as $\varepsilon \to 0$.

\begin{theorem}\label{PenalProbEstThm}
The solution $u_\varepsilon$ of the penalised problem \eqref{PenalProb} converges to the solution $u$ of the obstacle problem \eqref{ObsProb} strongly in $W^{s,p}_0(\Omega)$ as $\varepsilon \to 0$, with the estimate 
\begin{equation}\label{PenalProbEst}[u_\varepsilon-u]_{s,p}\leq  \begin{cases}C_p \varepsilon^{1/p} &\text{ for }p\geq2,\\
C_p \varepsilon^{1/2} &\text{ for }1<p<2,\end{cases}\end{equation} where \begin{equation}\label{Cp}C_p=\begin{cases}2^{-2/p}C_\theta^{1/p}\norm{ f^- }^{1/p}_{L^1(\Omega)} &\text{ for }p\geq2,\\
\frac{C_\theta^{1/2}}{(p-1)^{1/(2p-2)}}2^{\frac{-p^2-2p+4}{2p}}\norm{ f^- }_{L^1(\Omega)}^{1/2}\left(C_{SP}\norm{f^+}_{L^{p^\natural_s}(\Omega)}^{1/(p-1)}\right)^{1-p/2}&\text{ for }1<p<2.\end{cases}\end{equation} 
    
\end{theorem}

\begin{proof}

We first show that $u_\varepsilon\geq0$, so that the solution $u_\varepsilon\in \mathbb{K}^s$ for each $\varepsilon>0$. 
Indeed, taking $v=u_\varepsilon^-=(-u_\varepsilon)^+\geq0$ in \eqref{PenalProb}, we have \begin{align*}\langle -\Delta_p^s u_\varepsilon,u_\varepsilon^-\rangle = -\int_\Omega f^- \theta_\varepsilon (u_\varepsilon)  u_\varepsilon^- +\int_\Omega f^+  u_\varepsilon^-  \geq -\int_\Omega f^- \theta_\varepsilon (u_\varepsilon)  u_\varepsilon^- =0.\end{align*} The last equality is true because either $u_\varepsilon>0$ which gives $ u_\varepsilon^- =0$, or $u_\varepsilon\leq0$ which gives $\theta_\varepsilon (u_\varepsilon) =0$ by the construction of $\theta$ in \eqref{ThetaDef}--\eqref{ThetaDef2}, thus implying $\theta_\varepsilon (u_\varepsilon)  u_\varepsilon^- =0$. By the T-monotonicity of $-\Delta_p^s$, we conclude $ u_\varepsilon^- =0$, i.e. $u_\varepsilon\in \mathbb{K}^s$ for any $\varepsilon>0$.

Furthermore, we can show that $u_\varepsilon\geq0$ converges weakly in $W^{s,p}_0(\Omega)$ as $\varepsilon\to0$ to some $u$, which by uniqueness, is the solution of the obstacle problem. Indeed, the argument is classical, but we include it here for completeness. Taking $v=w-u_\varepsilon$ in \eqref{PenalProb} for arbitrary $w\in\mathbb{K}^s$, recalling $f^+=f+f^-$ and \eqref{ThetaDef2}, we have 
\begin{align}\langle -\Delta_p^su_\varepsilon,w-u_\varepsilon\rangle&=\int_\Omega f(w-u_\varepsilon)+\int_\Omega f^- [1-\theta_\varepsilon (u_\varepsilon) ](w-u_\varepsilon)\nonumber
\\&\geq\int_\Omega f(w-u_\varepsilon)-\varepsilon\int_\Omega f^- [1-\theta_\varepsilon (u_\varepsilon) ]\frac{u_\varepsilon}{\varepsilon}\nonumber
\\&\geq\int_\Omega f(w-u_\varepsilon)-\varepsilon C_\theta\int_\Omega f^- \label{Eq10}\end{align} since $ f^- ,1-\theta_\varepsilon$ and $w\geq0$ for $w\in\mathbb{K}^s_\psi$. 

By the monotonicty of $-\Delta_p^s$, for any $w\in\mathbb{K}^s$,
\[\langle -\Delta_p^s w,w-u_\varepsilon \rangle \geq \langle -\Delta_p^s u_\varepsilon,w-u_\varepsilon \rangle \geq\int_\Omega f(w-u_\varepsilon)-\varepsilon C_\theta\int_\Omega f^- .\] Since $[u_\varepsilon]_{s,p}$ is bounded independently of $\varepsilon>0$, at the limit $\varepsilon\to0$, letting $u$ be a weak limit of $u_\varepsilon$, we have
\[\langle -\Delta_p^s w,w-u \rangle  \geq \int_\Omega f(w-u).\] Now take $w_\delta = u+\delta (v-u)$ for any $v\in\mathbb{K}^s$ and $0<\delta<1$ . Then, 
dividing by $\delta$ and letting $\delta\to0$, we find 
\[\langle -\Delta_p^s u,v-u \rangle  \geq \int_\Omega f(v-u) \quad \forall v\in\mathbb{K}^s,\] so $u$ satisfies the obstacle problem \eqref{ObsProb}.

Now, taking $w=u$ in \eqref{Eq10}, we obtain \[\langle -\Delta_p^su_\varepsilon -f,u-u_\varepsilon\rangle\geq -\varepsilon C_\theta\int_\Omega f^- ,\] which added to the inequality of the original obstacle problem \eqref{ObsProb} with $v=u_\varepsilon\in\mathbb{K}^s_\psi$, by recalling \ref{LgCoerciveEq}, implies
 
\begin{align*}\varepsilon C_\theta\int_\Omega f^- &\geq\langle -\Delta_p^su_\varepsilon- (-\Delta_p^s)u,u_\varepsilon-u\rangle
\geq 
\begin{cases} 2^{1-p}[u_\varepsilon-u]_{s,p}^p&\text{ if }p\geq2,\\
(p-1)2^{\frac{p^2-4p+2}{p}}\dfrac{[u_\varepsilon-u]_{s,p}^2}{\left([u_\varepsilon]_{s,p}+[u]_{s,p}\right)^{2-p}}&\text{ if }1<p<2.
\end{cases}\end{align*} 
In order to complete the proof in the second case $1<p<2$ , since $u_\varepsilon$ solves \eqref{PenalProb}, taking the test function to be $u_\varepsilon$ and applying H\"older's inequality, we have that  
\begin{equation}\label{Eq11}(p-1)2^{-\frac{p^2-4p+2}{p}}[u_\varepsilon]^{p-1}_{s,p}\leq C_{SP}\norm{f^+}_{L^{p^\natural_s}(\Omega)}\end{equation} since $0\leq\theta_\varepsilon\leq1$, where $p^\natural_s=(p^*_s)'$ is the conjugate fractional Sobolev exponent with Sobolev-Poincar\'e constant $C_{SP}$ for $p^*_s$. 

\end{proof}

\begin{remark}
    
The results of Theorem \ref{PenalProbEstThm} are still valid in the $s=1$ case, by replacing \eqref{PenalProbEst} with the estimate
\begin{equation}\label{1-PenalProbEst}\norm{\nabla( u_\varepsilon-u)}_{L^{p}(\Omega)}\leq  \begin{cases}C_p \varepsilon^{1/p} &\text{ for }p\geq2,\\
C_p \varepsilon^{1/2} &\text{ for }1<p<2\end{cases}\end{equation}
with the same constant $C_p$ given by \eqref{Cp}, where $C_{SP}$ in the case $1<p<2$ is the Sobolev constant for $W^{1,p}_0(\Omega)$. This is the result of Theorem 2 of \cite{JFR2005j1},  but the estimates \eqref{PenalProbEst} for the nonlocal problems are new in the case $p\neq 2$, as the case $p=2$ is Theorem 4.1 in \cite{FracObsRiesz}. Moreover, while the constant $C_p$ is independent of $s$ for $p\geq 2$, in the case $1<p<2$ it may be replaced by a $C_p'>0$, also independent of $s>s_0>0$ if $f^+\in L^{p'}(\Omega)$. Indeed, using Poincaré's inequality, we may then replace the right hand side of \eqref{Eq11} by $c_P\norm{f^+}_{L^{p'}(\Omega)}$ in order to replace $C_p$ by $C_p'$ in \eqref{PenalProbEst}:
\begin{equation}\label{PenalProbEstIndepS}C_p'=\begin{cases}2^{-2/p}C_\theta^{1/p}\norm{ f^- }^{1/p}_{L^1(\Omega)} &\text{ for }p\geq2,\\
\frac{C_\theta^{1/2}}{(p-1)^{1/(2p-2)}}2^{\frac{-p^2-2p+4}{2p}}\norm{ f^- }_{L^1(\Omega)}^{1/2}\left(c_P\norm{f^+}_{L^{p'}(\Omega)}^{1/(p-1)}\right)^{1-p/2}&\text{ for }1<p<2.\end{cases}\end{equation} 
\end{remark}

\begin{remark}\label{rem:NonzeroObstacle}
    We can in fact adapt also the penalised problem 
    \eqref{PenalProb} for the obstacle problem with the convex set \eqref{ConvexSetDef} replaced by
    \begin{equation}\label{ConvexSetPsi}\mathbb{K}^s_\psi:=\{v\in W^{s,p}_0(\Omega):v\geq \psi \text{ a.e. in }\Omega \}\end{equation}
    with a nonzero obstacle $\psi\in W^{s,p}(\mathbb{R}^n)$, such that $\psi\leq 0$ in $\mathbb{R}^n\backslash\Omega$ and $-\Delta^s_p\psi\in L^{p^\natural_s}(\Omega)$. By replacing $f^+$ 
     with $f+\zeta$ for any $f, \zeta\in L^{p^\natural_s}(\Omega)$, such that $\zeta\geq(-\Delta^s_p\psi-f)^+\in L^{p^\natural_s}(\Omega)$, then the convergence of $u_\varepsilon\to u$ in $W^{s,p}_0(\Omega)$ holds with the same estimate \eqref{PenalProbEst},
 where now $C_p$ is given by
    \[C_p=\begin{cases}2^{-2/p}C_\theta^{1/p}\norm{\zeta}^{1/p}_{L^1(\Omega)} &\text{ for }p\geq2,\\
    \frac{C_\theta^{1/2}}{(p-1)^{1/(2p-2)}}2^{\frac{-p^2-2p+4}{2p}}\norm{\zeta}_{L^1(\Omega)}^{1/2}\left(C_{SP}\norm{f+\zeta}_{L^{p^\natural_s}(\Omega)}^{1/(p-1)}\right)^{1-p/2}&\text{ for }1<p<2.\end{cases}\]

\end{remark}

As a consequence of this approximation, we can show in the next theorem that in fact the solution of the obstacle problem \eqref{ObsProb} is a solution of a nonlinear fractional Dirichlet problem with a bounded second hand term involving a quasi-characteristic function $\vartheta^s=\vartheta^s(u^s)$ (see \eqref{ThetaLimit}):
\begin{equation}\label{EllipticProbEq}-\Delta_p^su^s=f^+-f^- \vartheta^s\quad\text{ in }\Omega,\quad\quad u^s=0\quad\text{ in }\mathbb{R}^n\backslash\Omega.\end{equation}

\begin{theorem}\label{LocalRegEllipNonlin1}
Let $0<s<1$, $1<p<\infty$ and $\Omega\subset\mathbb{R}^n$ be a bounded Lipschitz domain. Let $u^s$ denote the solution of the obstacle problem \eqref{ObsProb} with $f\in L^\infty(\Omega)$. 
Then, the unique solution to the nonlocal $p$-obstacle problem \eqref{ObsProb} is such that $u^s\in W^{s,p}_0(\Omega)\cap C^\alpha(\bar{\Omega})$, i.e.  $u^s$ is globally H\"older continuous in $\Omega$, for some $0<\alpha<1$ depending on $n$, $p$ and $s$, and satisfies \eqref{EllipticProbEq} where $\vartheta^s=\vartheta^s(u^s)$ is such that
\begin{equation}\label{ThetaLimit}
    0\leq \chi_{\{u^s>0\}} \leq \vartheta^s\leq1 \text{ a.e. in }\Omega.
\end{equation} 

Moreover, for a fixed $0<s_0<1$, $u^s\in W^{s,p}_0(\Omega)\cap C^{\beta}_{loc}(\Omega)$ uniformly in $s$, $s_0\leq s<1$, i.e.  $u^s$ is locally H\"older continuous in $\Omega$, for some $0<\beta\leq s_0$ depending only on $n$ and $p$. This means that for every open subset $\omega\Subset\Omega$ there exists a constant $C_\omega>0$ depending only on $n$, $p$ and $\norm{f}_{L^\infty(\Omega)}$ and independent of $s$, such that the solution $u^s$ of the nonlocal $p$-obstacle problem \eqref{ObsProb} satisfies
    \begin{equation}
    u^s\in C^\beta(\bar{\omega})\quad\text{ with }\quad  \norm{u^s}_{C^\beta(\bar{\omega})}\leq C_\omega.
    \end{equation}
\end{theorem}

\begin{proof}
    In fact we can apply the global H\"older estimates recalled in Theorem \ref{SemiLinearRegThm}
    to the solution $u^s_\varepsilon\in W^{s,p}_0(\Omega)$ of the nonlocal elliptic Dirichlet $p$-Laplacian problem \begin{equation}\label{EllipRegEqPenal}-\Delta_p^su^s_\varepsilon=f^+-f^- \theta_\varepsilon (u_\varepsilon^s)\quad\text{ in }\Omega,\quad\quad u^s_\varepsilon=0\quad\text{ in }\mathbb{R}^n\backslash\Omega,\end{equation} 
    in order to deduce \[\norm{u^s_\varepsilon}_{C^\alpha(\bar{\Omega})}\leq C_\Omega\] independently of $\varepsilon$. Then we let $\varepsilon\to 0$, and by the Arzela-Ascoli theorem,  $u^s_\varepsilon$ converges uniformly to $u^s$ and we obtain \[\norm{u^s}_{C^\alpha(\bar{\Omega})}\leq C_\Omega.\] Taking a weak-* limit in $L^{\infty}(\Omega)$ of $\theta_\varepsilon (u_\varepsilon)$, for a subsequence still denoted by $\varepsilon\to 0$\, i.e.\[\theta_\varepsilon (u_\varepsilon)\rightharpoonup \vartheta^s=\vartheta^s(u^s)\] by the definition of $\theta_\varepsilon$ and the pointwise convergence of $u_\varepsilon$, we conclude $\vartheta^s$ satisfies \eqref{ThetaLimit} at the limit.

    For the second result, we instead apply the robust local estimates of Theorem \ref{SemiLinearRegThm} to the solution $u^s_\varepsilon\in W^{s,p}_0(\Omega)$ of the nonlocal elliptic Dirichlet $p$-Laplacian problem \eqref{EllipRegEqPenal} to deduce \eqref{Robust} for $u^s_\varepsilon$ independently of $s$ and $\varepsilon$.

    
    
\end{proof}

Consequently, recalling that $0\leq \theta_{\varepsilon}\leq 1$, we directly obtain from the bounded penalised problem the Lewy-Stampacchia inequalities for the fractional obstacle problem (see  \cite{GMosconi} and Lemma 2.5 of \cite{IannizzottoMosconiSquassina2020JFA_FracPLapTMonotoneLewyStampacchia}).

\begin{remark}\label{LewyStampacchia}
Then, the unique solution $u^s$ to the obstacle problem \eqref{ObsProb} with non zero obstacle $\psi$, i.e. in $\mathbb{K}^s_\psi$ given by \eqref{ConvexSetPsi}, satisfies \begin{equation}\label{ObsPbH-sEqLewyStamp}f\leq -\Delta_p^su^s\leq f+(-\Delta^s_p\psi-f)^+=-\Delta^s_p\psi \vee f \quad \text{ a.e. in }\Omega.\end{equation}
Consequently, if $(-\Delta_p^s\psi-f)^+\in L^\infty(\Omega)$ and $f\in L^\infty(\Omega)$, then $-\Delta_p^su^s\in L^\infty(\Omega)$ and we can then write \eqref{ObsProb} in the form 
\[u^s\in\mathbb{K}^s_\psi:\quad\int_\Omega(-\Delta^s_pu^s-f)(v-u)\geq0\quad\forall v\in\mathbb{K}^s_\psi.\]

\end{remark}

\section{Stability of the Obstacle Problem as $s\to\sigma\leq 1$}\label{sect:NonlocalLgSConvg}

Given the nonlinear nonlocal obstacle problem for $0<s<1$ as well as the classical nonlinear obstacle problem for $s=1$, we consider now the behaviour as $s\to\sigma\leq 1$. This follows and extends to the case $\sigma<1$ the approach of \cite{BonderSalort2020StabilityNonlocal}, adapted to the case of the obstacle problem. 

We start by generalising Lemma \ref{BonderSalortConvgLemma}.
\begin{lemma}\label{BonderSalortConvgResult}
    For $0<s_k<1$, let $s_k\to\sigma \leq 1$ and $v_k\in W^{s_k,p}_0(\Omega)$ be such that $\sup_{k\in\mathbb{N}}[v_k]_{s_k,p}^p<\infty$ and $v_k\to v$ strongly in $L^p(\Omega)$. Then, 
    \[\langle-\Delta_p^{s_k}u,v_k\rangle\to\langle-\Delta_p^\sigma u,v\rangle\quad \text{ for every }u\in W^{1,p}_0(\Omega).\]
\end{lemma}

\begin{proof}
    The case where $\sigma=1$ and $s_k\nearrow1$ is given in Lemma 2.8 of \cite{BonderSalort2020StabilityNonlocal}.

    It remains to show the case $\sigma<1$. We first show that the semi-norm $[\cdot]_{s,p}$ is continuous with respect to $s$, following the proof given in Proposition 2.1 of \cite{HitchhikerGuide} for the case where the $W^{s,p}$-norm is instead defined with the double integrals in $\Omega$. We first decompose the semi-norm into two parts
    \begin{multline*}
    [u]_{s,p}^p= \int_{\mathbb{R}^n}\int_{\mathbb{R}^n} \frac{|u(x)-u(y)|^p}{|x-y|^{n+sp}} \,dx\,dy \\\leq \int_{\mathbb{R}^n}\int_{\mathbb{R}^n\,\cap\,\{|x-y|\geq 1\}}\frac{|u(x)-u(y)|^p}{|x-y|^{n+sp}} \,dx\,dy + \int_{\mathbb{R}^n}\int_{\mathbb{R}^n\,\cap\,\{|x-y|< 1\}}\frac{|u(x)-u(y)|^p}{|x-y|^{n+sp}} \,dx\,dy = I_1+ I_2.
    \end{multline*}
    For $I_1$, observe that 
    \begin{align*}
    \int_{\mathbb{R}^n}\int_{\mathbb{R}^n \cap \{|x-y|\geq 1\}} \frac{|u(x)|^p}{|x-y|^{n+sp}} \,dx\,dy  &\leq  \int_{\mathbb{R}^n} \left( \int_{|z|\geq 1} \frac{1}{|z|^{n+sp}}\,dz  \right) |u(x)|^p\,dx\\
    &\leq  C(n,s,p) \|u\|^p_{L^p(\Omega)}\,,
    \end{align*}
    since the kernel $1/|z|^{n+sp}$ is integrable for $n+sp>n$. Therefore,
    \begin{equation}
    I_1 \leq 2^{p-1}\int_{\mathbb{R}^n}\int_{\mathbb{R}^n\,\cap\{|x-y|\geq 1\}}\frac{|u(x)|^p+|u(y)|^p}{|x-y|^{n+sp}} \,dx\,dy \leq 2^{p} C(n,s,p) \|u\|^p_{L^p(\Omega)}.
    \end{equation}

    On the other hand, for $s'>s$, we have
    \[
    I_2 \leq  \int_{\mathbb{R}^n}\int_{\mathbb{R}^n\cap\,\{|x-y|< 1\}}\frac{|u(x)-u(y)|^p}{|x-y|^{n+s'p}} \,dx\,dy.
    \]

Therefore, 
\begin{equation}\label{SemiNormContinuity}
    [u]^p_{s,p} \leq 2^{p} C(n,s,p) \|u\|^p_{L^p(\Omega)}+ \int_{\mathbb{R}^n}\int_{\mathbb{R}^n} \frac{|u(x)-u(y)|^p}{|x-y|^{n+s'p}} \,dx\,dy \leq C'(n,s,p) [u]^p_{s',p}.
\end{equation}

Now for any sequence $\{s_k\}_k \to s$, let $s':=\sup_k s_k$. Then, from the above estimates, the integrands in the norms $[u]_{s_k,p}^p$ are dominated by the integrands in $[u]_{s',p}^p$, so we can apply the Lebesgue dominated convergence theorem to obtain the continuity of $s \mapsto [u]_{s,p}$ for any $u\in W^{s',p}(\Omega)$ with $s'>s$.

On the other hand, by the assumptions on $v_k$, we can apply Fatou's lemma to obtain 
\[\liminf_k[v_k]_{s_k,p}^p= \liminf_k\int_{\mathbb{R}^n}\int_{\mathbb{R}^n} \frac{|v_k(x)-v_k(y)|^p}{|x-y|^{n+s_kp}} \,dx\,dy\geq \int_{\mathbb{R}^n}\int_{\mathbb{R}^n} \liminf_k\frac{|v_k(x)-v_k(y)|^p}{|x-y|^{n+s_kp}} \,dx\,dy=[v]_{\sigma,p}^p,\] so $v\in W^{\sigma,p}_0(\Omega)$.

By the continuity of the $W^{s,p}$-norms, which corresponds to the continuity of the functionals $\mathcal{J}_{s,p}$ by \eqref{JFuncDef}, we can extend the proof of the Lemma 2.8 of \cite{BonderSalort2020StabilityNonlocal}. Indeed, as in the case $\sigma=1$, it suffices to show that 
\begin{equation} \label{desig.1}
\langle -\Delta_p^\sigma u, v\rangle  \leq    \liminf_{k\to\infty} \langle -\Delta_p^{s_k} u, v_k\rangle \quad \text{ for every } u\in W^{1,p}_0(\Omega)\subset W^{s,p}(\Omega) \text{ for every }s<1,
\end{equation} since we can consider $-u$ to get the reverse inequality. Now, applying Fatou's lemma again, we have that 
\[\mathcal{J}_{\sigma,p}(u+\tau v)\leq \liminf_{k\to\infty} \mathcal{J}_{s_k,p}(u+\tau v_k).\] Combining this with the continuity of the functionals $\mathcal{J}_{s,p}$ in \eqref{SemiNormContinuity} then gives 
\[\mathcal{J}_{\sigma,p}(u+\tau v) - \mathcal{J}_{\sigma,p}(u)\leq \liminf_{k\to\infty} \left(\mathcal{J}_{s_k,p}(u+\tau v_k) - \mathcal{J}_{s_k,p}(u)\right).\] Then, the result follows simply from applying Lemma \ref{BonderSalortConvgLemma}, dividing by $\tau>0$ and letting $\tau\to0$.
\end{proof}

Consequently, we have the equivalent of Theorem \ref{SemiLinearConvgThm} for any $\sigma$, $0<\sigma<1$, as a corollary, which proof is similar.

\begin{corollary}\label{SemiLinearConvgCor}
    Suppose that the Carathéodory function $F$ is non-increasing in $z$ and satisfies \eqref{FGrowth} with $f\in L^{p'}(\Omega)$. For $0<s_0<\sigma<1$, if $s\to\sigma$ then $u^s\to u^\sigma$ in $W^{r,p}_0(\Omega)$, for any $0\leq r<\sigma$, where $u^s\in W^{s,p}_0(\Omega)$ and $u^\sigma\in W^{\sigma,p}_0(\Omega)$ are the solutions to \eqref{DirichletExistEq} corresponding to $s$ and $\sigma$ respectively.

\end{corollary}

As a consequence, we have the equivalent of Theorem \ref{SemiLinearConvgThm} and Corollary \ref{SemiLinearConvgCor} for the fractional $p$-obstacle problem.

\begin{theorem}\label{LpConvgThm}
Let $u^s$, for $0<s<1$, be the solution to the nonlinear nonlocal obstacle problem \begin{equation}\label{LpObsProbConvgS}u^s\in \mathbb{K}^s:\quad \langle -\Delta_p^su^s-f,v-u^s\rangle\geq0\quad\forall v\in \mathbb{K}^s.\end{equation} Then, $u^s\to u^\sigma$ in $W^{r,p}_0(\Omega)$ as $s\to\sigma$, for any $0\leq r<\sigma$, for every $0<\sigma\leq 1$, where $u^\sigma\in \mathbb{K}^\sigma$ solves uniquely the obstacle problems for $s=\sigma$. In the case of $\sigma=1$, $u^1$ solves the classical obstacle problem \begin{equation}\label{ClassicalObsProbLg}u^1\in \mathbb{K}^1:\quad\langle -\Delta_p u^1 -f,v-u^1\rangle\geq0\quad\forall v\in \mathbb{K}^1,\end{equation} 
where 
\[\mathbb{K}^1=\{v\in W^{1§,p}_0(\Omega):v\geq 0 \text{ a.e. in }\Omega\}.\]
\end{theorem}

\begin{proof}
For any fixed $\varepsilon>0$,  in the case of $\sigma=1$ we can apply Theorem \ref{SemiLinearConvgThm} and in the case of $0<\sigma<1$ the Corollary \ref{SemiLinearConvgCor} to take $s\to \sigma$ in the associated penalised problems \eqref{PenalProb}, i.e. in
    \begin{equation}\label{PenalProbS}
    u_\varepsilon^s\in W^{s,p}_0(\Omega):\quad\langle -\Delta_p^s u_\varepsilon^s,v\rangle+\int_\Omega f^- \theta_\varepsilon(u^s_\varepsilon) v=\int_\Omega  f^+ v, \quad\forall v\in W^{s,p}_0(\Omega).\end{equation} 
Then,  $u^s_\varepsilon\to u^\sigma_\varepsilon$ in $W^{r,p}_0(\Omega)$ as $s\to\sigma$, for any $0\leq r<\sigma$, where $u^\sigma_\varepsilon$ solves the approximated problem \eqref{PenalProbS} for $s=\sigma$, $0<\sigma\leq 1$.  

On the other hand, for all $\varepsilon>0$, both $[u^s_\varepsilon-u^s]_{s,p}$ and $[u^\sigma_\varepsilon-u^\sigma]_{\sigma,p}$ satisfy the estimate \eqref{PenalProbEst} with $C_p$ replaced by $C_p'$ given by \eqref{PenalProbEstIndepS}, which is independent of $s$ and $\sigma$. 

As a consequence, for an arbitrary $\delta>0$, we may choose an $s$ sufficiently close to $\sigma$ and an $\varepsilon_\delta>0$, such that
    \[[u^\sigma-u^s]_{r,p}\leq [u^\sigma-u^\sigma_{\varepsilon_\delta}]_{r,p}+[u^\sigma_{\varepsilon_\delta}-u^s_{\varepsilon_\delta}]_{r,p}+[u^s_{\varepsilon_\delta}-u^s]_{r,p}<\delta,\] for any fixed $r<\sigma$ such that $s>r$.  Indeed, by the estimate \eqref{PenalProbEst}, there exists a sufficiently small $\varepsilon_s>0$, such that 
    
    \[[u^s_{\varepsilon_s}-u^s]_{r,p}\leq C_{s,r}[u^s_{\varepsilon_s}-u^s]_{s,p}\leq C_{s,r} C'_p \varepsilon_s^{\frac{1}{p\vee 2}}< \delta/3, \]
where $C_{s,r}>0$ is the constant of the continuous inclusion $W^{s,p}_0(\Omega)\subset W^{r,p}_0(\Omega)$. Analogously, there exists an $\varepsilon_\sigma>0$, such that $[u^\sigma-u^\sigma_{\varepsilon_\sigma}]_{r,p}\leq C_{\sigma,r} C'_p \varepsilon_\sigma^{\frac{1}{p\vee 2}}< \delta/3$. Choosing $\varepsilon_\delta = \varepsilon_s\wedge\varepsilon_\sigma$, by the convergence $u^s_{\varepsilon_\delta}\to u^\sigma_{\varepsilon_\delta}$ in $W^{r,p}_0(\Omega)$ as $s\to\sigma$, we may also take $[u^\sigma_{\varepsilon_\delta}-u^s_{\varepsilon_\delta}]_{r,p}< \delta/3$, concluding the proof of the theorem.

\end{proof}

As a consequence of Theorem \ref{LocalRegEllipNonlin1} and Theorem \ref{LpConvgThm}, we obtain as a corollary the following uniform convergence.
\begin{corollary}\label{LpConvgCorUniform}
    There exists a $\lambda$, $0<\lambda<\beta$, such that 
    \begin{equation}u^s\to u^1\quad\text{ in }C^\lambda(\bar{\omega})\quad\text{ as } s\nearrow1, \quad \text{ for every open subset $\omega\Subset\Omega$,}\end{equation} and therefore the convergence is uniform in the interior of $\Omega$.
\end{corollary}

As a corollary of this uniform convergence, by Theorem \ref{LocalRegEllipNonlin1}, whenever $u^s>0$ we have $\vartheta^s=1$, for $0<s\leq1$, and by \eqref{EllipticProbEq} the following result follows immediately.
\begin{corollary}\label{LpConvgCorCoin}
    For every measurable subset $\omega\Subset\Omega_+^1=\{x\in\Omega:u^1(x)>0\}$ and $s$ sufficiently close to 1 we have $\omega\subset\Omega_+^s=\{x\in\Omega:u^s(x)>0\}$ and $-\Delta^s_pu^s=f=-\Delta_pu^1$ a.e. in $\omega$.
\end{corollary}

\begin{remark}\label{LinConvgRemark}
In the case $p=2$ where $-\Delta_p^s$ corresponds to the linear nonlocal operator $-\Delta^s$ for $u,v\in H^s_0(\Omega)$ in Lipschitz domains $\Omega\subset\mathbb{R}^n$, the result was given in Theorem 3.8 of \cite{FracObsRiesz}.
\end{remark}

\begin{remark}
    An interesting extension will be to obtain the convergence $s\to 1$ for general non-homogeneous nonlinear kernels \[\int_{\mathbb{R}^n}\int_{\mathbb{R}^n}\frac{|u(x)-u(y)|^{p-2}(u(x)-u(y))(v(x)-v(y))}{|x-y|^{n+sp}}K_s(x,y)\,dx\,dy.\] A major difficulty is that in the nonlocal case, the anisotropy given by $K_s(x,y)$ is defined in two variables, while the anisotropy of the local $p$-Laplacian 
    \[\int_\Omega A(x)|\nabla u|^{p-2}\nabla u \cdot \nabla v \, dx\] is defined in one variable. In some cases, with suitable assumptions on $K_s(x,y)$, as in \cite{BonderSalort2022AsymptoticSFracEnergies}, it is possible to establish a $\Gamma$-convergence of energy functionals but it is not known how to write explicitly the expression for the limit energy $\mathcal{J}_{1,p}$ . Another interesting extension in the nonlinear case is to consider the non-homogeneous Dirichlet problem, as the Dirichlet boundary condition is defined nonlocally a.e. in $\mathbb{R}^n\backslash\Omega$, for $0<s<1$, and defined locally by the trace on $\partial\Omega$ for $s=1$ (see \cite{Kassmann2023}).
\end{remark}

\section{Stability of the $s$-Coincidence Sets in Measure as $s\nearrow1$}

Since the solutions $u_s$ are unique we have the stability of $u_s\to u_\sigma$ for the whole sequence $s\to\sigma\leq1$. on the other hand, the quasi-characteristic functions $\vartheta_s$ are not known to be unique in general, so we only prove a weaker stability, which is our first result in this section.

\begin{proposition}\label{ConvgQ}
    Let $f\in L^\infty(\Omega)$,   and, up to a set of Lebesgue measure zero, assumed
    \begin{equation}\Lambda=\{x\in\Omega:f(x)<0 \}\neq\emptyset.\end{equation} Let $(u^s, \vartheta^s)$ satisfy \eqref{EllipticProbEq} and \eqref{ThetaLimit}. Then, as $s\to\sigma\leq1$, there exists a subsequence $s_k$ and a function $\vartheta^*$, such that $0\leq\vartheta^*\leq 1$ and 
    \begin{equation}\label{qConvg}
        \vartheta^{s_k} \rightharpoonup\vartheta^* \text{ in $L^\infty(\Omega)$-weak$^*$} \quad \text{ as } \quad s_k\to\sigma\end{equation}
    with 
    \begin{equation}\vartheta^*=\vartheta^\sigma \text{ a.e. in }\Lambda,\end{equation}
    where $(u^\sigma, \vartheta^\sigma)$ 
    satisfies \eqref{EllipticProbEq} and \eqref{ThetaLimit}.
\end{proposition}

\begin{proof}
   By \eqref{EllipticProbEq} and \eqref{ThetaLimit}, there exists a function $\varphi^s\in L^\infty(\Omega)$ such that
    \begin{equation}\label{ConvgCoin3}
    -\Delta_p^su^s = f^+ - f^-\vartheta^s=\varphi^s \quad\text{ a.e. in }\Omega.\end{equation}
 Since $0\leq\vartheta_s\leq1$, we can take a subsequence $s_k\to\sigma$ such that \eqref{qConvg} holds for some function function $\vartheta^*$, such that $0\leq\vartheta^*\leq 1$. 
    
    Furthermore, still denoting $s_k$ by $s$, $\varphi^s$ converges to  $\varphi^*=f^+ - f^-\vartheta^*$ in $L^{\infty}(\Omega)$-weak$^*$. Passing to the limit in \eqref{ConvgCoin3}, by Lemma \ref{BonderSalortConvgResult} and Theorem \ref{LpConvgThm}, we have 
    \begin{equation}\label{ConvgSigma}
     -\Delta_p^su^s\rightharpoonup - \Delta^\sigma_pu^\sigma  \quad\text{ in }\quad L^\infty(\Omega)-weak^* \quad\text{ as } s\to \sigma, \quad  0<\sigma\leq 1.    
    \end{equation}
    
    This is because 
    \[0\leq \langle -\Delta^s_pu^s,u^s-v\rangle-\langle-\Delta^s_pv,u^s-v\rangle\to\langle\varphi^*,u^\sigma-v\rangle-\langle-\Delta_p^\sigma v,u^\sigma-v\rangle\quad\forall v\in C_c^\infty(\Omega),\] and the convergence follows by Lemma \ref{BonderSalortConvgResult}. Since $C_c^\infty(\Omega)$ is dense in $W^{\sigma,p}_0(\Omega)$, we can take $v=u^\sigma \pm \delta w$ in this inequality, with arbitrary $w\in W^{\sigma,p}_0(\Omega)$ and let $\delta \to 0$, in order to conclude 
    \begin{equation}\label{ConvgCoin4}
    -\Delta_p^\sigma u^\sigma = \varphi^*=f^+ - f^-\vartheta^*\quad\text{ a.e. in }\Omega.\end{equation}
    
   Now, comparing \eqref{EllipticProbEq} for $s=\sigma$ with \eqref{ConvgCoin4} we obtain
    \begin{equation}\vartheta^* f^- = \vartheta^\sigma f^- \quad\text{ a.e. in }\Omega,\end{equation}
    concluding the proof of the proposition. 
\end{proof}

Using the Lewi-Stampacchia estimates, if we assume  the nondegeneracy assumption that the Lebesgue measure of the set  $\{f=0\}$ vanishes, we can also obtaining the weak stability of the quasi-characteristic functions $\vartheta_s$, as a consequence of the previous proposition.

\begin{corollary}
    If $f\neq0$ a.e. in any open subset $\omega\subseteq\Omega$, then $\vartheta^s \rightharpoonup\vartheta^\sigma$ in $L^\infty(\omega)$-weak$^*$  as $s\to\sigma$, for $0<\sigma\leq 1$.
\end{corollary}
\begin{proof}
    The Lewi-Stampacchia inequalities, $f\leq -\Delta_p^su^s\leq f^+$ a.e. in $\Omega$, imply that $\vartheta^s=0$ a.e. in $\{f>0\}$, where $f=f^+$, for every $0<s\leq 1$. Consequently, also $\vartheta^*=0$ in $\{f>0\}$, which, with the result of the proposition yielding $\vartheta^*=\vartheta^\sigma$ in $\{f<0\}$ and the assumption  $meas(\{f=0\}\cap \omega)=0$, clearly implies $\vartheta^*=\vartheta^\sigma$ also a.e. in $\omega$, concluding the proof.
\end{proof}

In the local case $s=1$, denoting now $u^1=u$ in the obstacle problem for the $p$-Laplacian, it is known that, locally in any non-empty open subset $\omega\Subset\Omega$, $f\in L^\infty(\Omega)$ is such that \begin{equation}\label{ConvgCoin1f}f(x) \leq -\lambda <0\quad\text{ for }\lambda>0, \quad \text{ a.e. }x \in \omega\Subset\Omega,\end{equation}
which implies that the free boundary $\partial\{u>0\}\cap\omega$ has Lebesgue measure zero (see \cite{Caffarelli1981pLapCoincidenceStability} for $p=2$, \cite{LeeShahgholian2003pLapCoincidenceStability} for $p>2$ and \cite{ChallalLyaghfouriRodrigues2012pLapCoincidenceStability} for $1<p<2$).
As a consequence of Theorem \ref{LocalRegEllipNonlin1}, with this local nondegeneracy of the limit problem we
have that in fact the $\vartheta^1=\vartheta$ is a.e. uniquely determined by a characteristic function, i.e.

\begin{equation}\label{qi}
    0\leq \vartheta=\chi_{\{u>0\}}=1-\chi_{\{u=0\}}\leq1 \quad\text{ a.e. in }\omega,
\end{equation}
and we can write the equation for $u=u^1$ 
\begin{equation}\label{ConvgCoin0}
    -\Delta_p u = f - \chi_{{}_I} f^-\quad\text{ a.e. in }\omega\Subset\Omega,
\end{equation}
in terms of the characteristic function $\chi_{{}_I}$ of the coincidence set 
\[I=\{u=0\}.\]
Then, using the method of Theorem 6:6.1 of \cite{ObstacleProblems}, used in a different framework with linear operators and extended to the $p$-Laplacian type operators in \cite{JFR2005j1}, we can prove the following convergence of the coincidence sets.

\begin{theorem}
     Suppose $f\in L^\infty(\Omega)$ satisfies \eqref{ConvgCoin1f}, so that the free boundary $\partial\{u>0\}\cap\omega$ in the limit case $s=1$ has Lebesgue measure zero. Let $u^s$ be the solution of \eqref{LpObsProbConvgS} and $u$ be the limit solution of the local \eqref{ClassicalObsProbLg} corresponding to the limit $s\nearrow1$. Then, the respective coincidence sets of the nonlocal obstacle problems, defined by \[I^s= \{u^s=0\},\] converge locally in $\omega$ towards the coincidence set of the local obstacle problem, i.e. $I^s_\omega=I^s\cap{\omega}$ converges to $I_\omega=I\cap{\omega}$, 
   in the sense that their characteristic functions satisfy
    \begin{equation}\label{ConvgCoin5}\chi_{{}_{I^s_\omega}} \to \chi_{{}_{I_\omega}}\quad \text{ in }L^r(\Omega), \quad \forall 1\leq r<\infty.\end{equation} 
    Here $\chi_{{}_{I_\omega}}(x) = 1$ for $x\in I_\omega$ and $0$ otherwise.

\end{theorem} 

\begin{proof}
    As in Proposition \ref{ConvgQ}, for every $0<s<1$, there exists a function $\vartheta^s\in L^\infty(\Omega)$ such that
    \begin{equation}\label{ConvgCoin6}
    -\Delta_p^su^s = f^+ - f^-\vartheta^s \quad\text{ a.e. in }\omega.\end{equation} Recalling \eqref{ThetaLimit}, i.e. 
    \[0\leq \chi_{\{u^s>0\}}\leq\vartheta^s \leq 1,\] and we have 
    \begin{equation}\label{ConvgCoin7}0\leq q^s = 1- \vartheta^s \leq 1- \chi_{\{u^s>0\}}=\chi_{\{u^s=0\}}= \chi_{{}_{I^s_\omega}}\leq 1\quad\text{ in }\omega.\end{equation}
    Taking a subsequence and relabelling it as $s\to1$, we have 
    \begin{equation}\label{ConvgCoin8}q^s\rightharpoonup q \quad \text{ and }\quad \chi_{{}_{I^s_\omega}}\rightharpoonup\chi_*\quad \text{ in }L^\infty(\Omega)\text{-weak}^*\end{equation}
    for functions $q,\chi_*\in L^\infty(\Omega)$, which, from \eqref{ConvgCoin7}, must satisfy 
    \begin{equation}\label{ConvgCoin9}0\leq q \leq \chi_*\leq 1\quad\text{ a.e. in }\omega.\end{equation}

    Passing to the limit in \eqref{ConvgCoin6}, by Proposition \ref{ConvgQ}, we have 
    \begin{equation}\label{ConvgCoin10}
    -\Delta_p u - f = q f^-\quad\text{ a.e. in }\omega.\end{equation}

    Since $f(x)<0$ for $x\in \omega$, recalling \eqref{qi} and \eqref{ConvgCoin0}, we obtain
    \begin{equation}\label{ConvgCoin12}q f^- = (1-\vartheta) f^- =\chi_{{}_{I}} f^- \quad\text{ a.e. in }\omega.\end{equation} 
    Therefore, in $\omega\cap\{f<0\}$, $q=\chi_{{}_{I_\omega}}$ a.e. in $\omega$. Comparing with \eqref{ConvgCoin9}, we have
    \begin{equation}\label{ConvgCoin13}\chi_{{}_{I_\omega}}\leq \chi_*\quad\text{ a.e. in }\omega.\end{equation}

    On the other hand, since $u^s\to u$ in $L^p(\Omega)$, by \eqref{ConvgCoin8}, 
    \[0=\int_\omega\chi_{{}_{I^s_\omega}}u^s\,dx\to\int_\omega\chi_*u\,dx=0,\]
    hence $\chi_*u=0$ a.e. in $\omega$. This means that $\chi_*=0$ if $u>0$. Since $0\leq\chi_*\leq1$, we have also 
    \begin{equation}\label{ConvgCoin14}\chi_{{}_{I_\omega}}\geq\chi_*\quad\text{ a.e. in }\omega.\end{equation} 

    Combining this with \eqref{ConvgCoin13}, we get 
    \begin{equation}\label{ConvgCoin15}\chi_{{}_{I_\omega}}=\chi_*\quad \text{ a.e. in }\omega\end{equation} and the whole sequence $\chi_{{}_{I^s_\omega}}\to\chi_{{}_{I_\omega}}$ converge first weakly, and, since they are characteristic functions, also strongly in $L^r(\Omega)$ for any $1\leq r<\infty$.

\end{proof}

\begin{remark}
In the linear case $p=2$, the assumption \eqref{ConvgCoin1f} can be relaxed to $f(x)\neq 0$ a.e. $x\in\omega\Subset\Omega$, which, by the $W^{2,q}_{loc}$ regularity of the classical obstacle problem \cite{ObstacleProblems}, is enough in order that the limit solution satisfies \eqref{ConvgCoin0}.
\end{remark}

\begin{remark}
The local convergence of the characteristic functions \eqref{ConvgCoin5} is a local convergence in the Lebesgue measure of the coincidence sets 
\[I^s_\omega \xrightarrow[]{\mathcal{L}} I_\omega \quad \iff \quad d_\mathcal{L}(I^s_\omega,I_\omega)\to 0\] as $s\to 1$, up to null measure sets, where 
\[d_\mathcal{L}(I^s_\omega,I_\omega)= meas(I^s_\omega \div I_\omega) = \int_\omega|\chi_{{}_{I^s}}-\chi_{{}_{I}}|^r\,dx, \quad\forall r\geq1\]
and $I^s_\omega \div I_\omega=(I^s_\omega\backslash I_\omega)\cup(I_\omega\backslash I^s_\omega)$ denotes the symmetric difference of sets.

\end{remark}

\begin{remark}
    Unlike the convergence result for the limit obstacle problem $s=1$, where it is known that the free boundary has zero Lebesgue measure, in general, it is not possible to obtain a similar convergence on the coincidence sets in the nonlocal cases when $s\nearrow\sigma<1$ because that nondegeneracy property is not known to hold in the nonlocal case of $\sigma<1$.

\end{remark}

\section{Convergence of the Free Boundaries in Hausdorff Distance as $s\nearrow1$}

The local strict negative forcing \eqref{ConvgCoin1f} also implies a known growth of the solution $u=u^1$ of the limit $p$-obstacle problem from the free boundary 
\[\Phi\cap\omega, \quad\text{where}\quad \Phi=\partial\{u>0\} \quad\text{and} \quad \omega\Subset \Omega.\] In this section we show that growth property, combined with the uniform convergence from Corollary \ref{LpConvgCorUniform},
implies the local convergence of the free boundaries $\Phi^s\cap \bar{\omega}\to\Phi\cap\bar{\omega}$ in Hausdorff distance $d_\mathscr{H}$ as $s\nearrow1$, whenever the local topological property \eqref{ConvgH2} of the limit coincidence set $I^1_{\bar{\omega}}=I^1\cap \bar{\omega}$ holds. 

We recall that the Hausdorff distance between the sets $A$ and $C$ is defined by 
\[d_\mathscr{H} (A,C) := \max \{\sup_{x\in A} d(x,C), \sup_{y\in C} d(y,A)\} = \inf\{\delta>0: B_\delta(A)\ni C\text{ and }B_\delta(C)\ni A\},\] 
where $B_\delta(A)=\bigcup_{x\in A}B_\delta(x)$ is the $\delta$-parallel body to the set $A$, for the ball $B_\delta(x)$ of radius $\delta$ centred at $x$. We also recall that the space $\mathscr{H} (\bar{\omega})$ of all compact subsets of the compact set  $\bar{\omega}\subset\mathbb{R}^n$ is a compact metric space for the Hausdorff distance $d_\mathscr{H}$.

We start by recalling, in the local case of $s=1$, the well known growth property of the solution from the free boundary, which is given, for instance, in Lemma 3.1 of \cite{KarpKilpelainenPetrosyanShahgholian2000JDERegFreeBdryPLap}, in Lemma 1 of \cite{Andersson2020SPMJRegFreeBdryPLap} for $p>2$, and as Proposition 3.1 in \cite{Z+Z+Z2013PIASRegFreeBdryQuasilinearElliptic} for general quasilinear elliptic operators. This growth condition, used in the linear case $p=2$ by Caffarelli in \cite{Caffarelli1981pLapCoincidenceStability}, still holds for the linear fractional Laplacian by Lemma 3.2 of \cite{BarriosFigalliRosoton2018AJMRegFreeBdryFracLap}, where no stability results are known for the free boundary.
\begin{lemma}\label{GrowthCondLemma}
    For the solution $u=u^1$ to the problem \eqref{ClassicalObsProbLg} with $f\in L^\infty(\Omega)$ and $s=1$, there exists a constant $r_1$ depending only on $p$ such that for each $z\in\bar{\Omega}_+=\overline{\{u>0\}}$ and $r\in(0,r_1)$ satisfying $B_r(z)\subset\Omega$, we have 
    \begin{equation}\label{GrowthCond}\sup_{B_r(z)\cap\Omega_+^1}u\geq C_1r^{\frac{p}{p-1}}\end{equation} 
    for some positive constant $C_1$ depending only on $p$ and $r_1$.

\end{lemma}

Our main result is as follows.
\begin{theorem}
    Let $u^s$ be the solution to \eqref{LpObsProbConvgS} for $0<s<1$ and $u=u^1$ be the solution to \eqref{ClassicalObsProbLg}. Assume $f\in L^\infty(\Omega)$ satisfies \eqref{ConvgCoin1f}, and the limit closed set $I_{\bar{\omega}}=I\cap \bar{\omega}$ ($I=I^1$), for any smooth open set $\omega \Subset\Omega$, is such that 
    \begin{equation}\label{ConvgH2}I_{\bar{\omega}}=\overline{int(I_{\bar{\omega}})}. \end{equation} Then the associated free boundaries $\Phi^s=\partial I^s\cap\Omega$ are such that \[\Phi^s\cap \bar{\omega}\to\Phi\cap\bar{\omega}\text{ in Hausdorff distance as $s\nearrow1$}.\]
\end{theorem}

\begin{proof}
    By Theorem \ref{LpConvgThm}, the solutions $u^s$ converge strongly to $u$ as $s\nearrow1$ in $W^{r,p}_0(\Omega)$ for $r<1$ and also uniformly in $C^\lambda(\bar{\omega})$ for some $\lambda$, $0<\lambda<\beta$, for every open subset $\omega\Subset\Omega$ by Corollary \ref{LpConvgCorUniform}. Therefore, for $s<1$ sufficiently close to $1$, there exists some constant $\delta$ with $2\delta<C_1$ such that 
    \[|u^s(x^s)-u(x^1)|\leq\delta r^\frac{p}{p-1},\] where $x^s,x^1\in B_r(z)\cap\omega_+$ are the points such that 
    \[u^s(x^s)=\sup_{B_r(z)\cap\omega_+}u^s, \qquad  u(x^1)=\sup_{B_r(z)\cap\omega_+} u,\] 
    where $z\in\bar{\omega}_+=\overline{\{x\in \omega:u(x)>0\}}$ and $C_1$ are as in Lemma \ref{GrowthCondLemma}.
    Then, by the choice of $x^s$ and $x^1$, 
    \[\sup_{B_r(z)\cap\omega_+}u^s=u^s(x^s)=u^s(x^s)-u(x^1)+u(x^1)\geq (C_1-2\delta)r^{\frac{p}{p-1}},\] since $u$ satisfies \eqref{GrowthCond}. Hence, $u^s$ also satisfies a growth condition of the form \eqref{GrowthCond}.

    The rest of the proof essentially follows that of Theorem 6:5.8 of \cite{ObstacleProblems}, but for completeness we give the main ideas of the proof. 

    Consider an arbitrary small open ball $B_\rho\subset\omega$ such that $B_\rho\cap\Phi^s=\emptyset$ for infinitely many $s$ as $s\nearrow1$. Then either $u^s=0$ or $u^s>0$ in $B_\rho$. For the first case, by the convergence $u^s\to u$, we have that $u=0$ in $B_\rho$. For the second case, since $u^s$ satisfies $-\Delta_p^s u^s=f$ in $B_\rho\cap\{u^s>0\}$, taking this to the limit $s\nearrow1$, we have that $-\Delta_pu=f$ in the interior of $B_\rho\cap\{u>0\}$, by Corollary \ref{LpConvgCorCoin}, and since $f$ satisfies \eqref{ConvgCoin1f} a.e. in $\omega$, one has $int(I)\cap B_\rho=\emptyset$. But by \eqref{ConvgH2}, $B_\rho\subset\{u>0\}$. Therefore, $B_\rho\cap\Phi=\emptyset$ for both cases, and 
    \[\inf\{\delta>0:\Phi\cap\bar{\omega}\subset B_\delta(\Phi^s\cap\bar{\omega})\}\to0 \quad \text{ as }s\nearrow1.\]

    On the other hand, let $\rho>0$ denote the radius of any open ball $B_\rho\subset\omega$ such that $B_\rho\cap\Phi=\emptyset$. For the case of $u>0$ in $B_\rho$, by the uniform convergence of $u^s$ to $u$ in Corollary \ref{LpConvgCorUniform}, for any $0<\rho'<\rho$ and for all $s$ close enough to 1, $u^s>0$ in $B_{\rho'}$ with $B_{\rho'}\cap\Phi^s=\emptyset$. For the case of $u=0$ in $B_\rho=B_\rho(x_0)$ for some $x_0\in I_{\bar{\omega}}$, suppose one can find a subsequence $y^s\in \Phi^s\cap B_{\rho'}(x_0)$ for some $\rho'>0$ with $\epsilon=\rho-\rho'>0$, such that $u^s(x^s)$ satisfies the uniform growth condition \eqref{GrowthCond} for some $x^s\in \{u^s>0\}\cap\partial B_\epsilon(y^s)\subset B_\rho(x_0)$. By the uniform convergence of $u^s\to u$ in $B_\rho$, as $s\nearrow1$, we have a contradiction with $u=0$ in $B_\rho$. 
    
    Therefore, we have that $B_\rho\cap\Phi^s=\emptyset$ for all $s$ close enough to 1, and it follows that 
    \[\inf\{\delta>0:\Phi^s\cap\bar{\omega}\subset B_\delta(\Phi\cap\bar{\omega})\}\to0 \quad \text{ as }s\nearrow1.\]
\end{proof}
The strictly negative local assumption \eqref{ConvgCoin1f} on the forcing $f$ may be relaxed to the more general nondegeneracy condition \eqref{ConvgH1}. In this case we may still prove a weaker local result, in any smooth open set $\omega \Subset\Omega$, on the convergence of the coincidence sets 
\[I^s_{\bar{\omega}}=I^s\cap \bar{\omega}\]
in Hausdorff distance $d_\mathscr{H}$ as $s\nearrow1$, under the very same limit local regularity \eqref{ConvgH2} on $I_{\bar{\omega}}$.

\begin{theorem}
    Let $u^s$ be the solution to \eqref{LpObsProbConvgS} for $0<s<1$ and $u=u^1$ be the solution to \eqref{ClassicalObsProbLg}. Assume $f\in L^\infty(\Omega)$ satisfies
    \begin{equation}\label{ConvgH1}f(x)\neq0\quad\text{ a.e. }x\in\omega\Subset\Omega,\end{equation} 
    and the limit closed set $I_{\bar{\omega}}$ satisfies \eqref{ConvgH2}. 
    Then \[I^s_{\bar{\omega}}\to I_{\bar{\omega}}\text{ in Hausdorff distance as $s\nearrow1$}.\]
\end{theorem}

\begin{proof}
    This follows also from the uniform convergence $u^s\to u$. The argument of the proof is similar to Theorem 6:6.4 of \cite{ObstacleProblems}, as, by the compactness of $\mathscr{H}(\bar{\omega})$, we can extract a subsequence of $\{I^s_{\bar{\omega}}\}_s$ such that for some $I^*\in \mathscr{H}(\bar{\omega})$, 
    \begin{equation}\label{ConvgH10}I^s_{\bar{\omega}}\to I^* \quad \text{ in }\mathscr{H}(\bar{\omega}),\end{equation} the conclusion follows if we show $I^*=I_{\bar{\omega}}$.

    As $d_\mathscr{H}(I^s_{\bar{\omega}},I^*)\to0$, for each $x\in I^*$, there exists an $x^s\in I^s_{\bar{\omega}}$ such that $x^s\to x$. Hence, $u^s\to u$ in $C^0(\bar{\omega})$ implies 
    \[|u(x)| = |u(x) - u^s(x^s)| \leq |u(x) - u(x^s)| + |u(x^s) - u^s(x^s)| \to 0,\] so that $x\in I_{\bar{\omega}}$ and $I^*\subset I_{\bar{\omega}}$.

    Conversely, to prove $I_{\bar{\omega}}\subset I^*$, we show that the assumption \eqref{ConvgH1} implies $I_{\bar{\omega}} = I^* \cup N$ with $int(N)=\emptyset$. Suppose that that exists an open set $W\subset I_{\bar{\omega}}\backslash I^*$ such that $W\neq\emptyset$. Since $u^s$ solves the obstacle problem \eqref{LpObsProbConvgS}, by Theorem \ref{LocalRegEllipNonlin1}, we have $-\Delta_p^su^s=f$ in $\{u^s>0\}$, and so
    \begin{equation}\label{ConvgH11}supp\left(-\Delta_p^su^s-f\right)\cap\bar{\omega}\subset I^s_{\bar{\omega}}\quad\text{ for every }0<s<1.\end{equation} Then by Theorem \ref{LpConvgThm} we have $\Delta_p^su^s\rightharpoonup\Delta_pu$, by  \eqref{ConvgSigma}, and using \eqref{ConvgH10} with Lemma \ref{ConvgLemma} below, from \eqref{ConvgH11},  in the limit we conclude
    \begin{equation}\label{ConvgH12}supp (-\Delta_pu-f) \cap\bar{\omega}\subset I^*\end{equation} 
    and hence 
    \begin{equation}\label{ConvgH13} -\Delta_pu = f \quad \text{ in }W\subset\omega\backslash I^*.\end{equation} 
    On the other hand, $W$ is open and is included in $I_{\bar{\omega}}=\{x\in\bar{\omega}: u=0\}$, so one has $-\Delta_pu=0$ in $W$. Consequently, $f=0$ in $W$ which contradicts \eqref{ConvgH1}. Then $int(I_{\bar{\omega}}) = int(I^*)$ and, since $I^*$ is closed, by the assumption \eqref{ConvgH2}, we have that 
    \[I_{\bar{\omega}} = \overline{int(I_{\bar{\omega}})} = \overline{int(I^*)} \subset I^*.\] This concludes the proof.
\end{proof}

\begin{lemma}[Lemma 6:6.4 of \cite{ObstacleProblems}]\label{ConvgLemma}
    In a compact set $\bar{\omega}\subset \Omega$, let $\{T_s\}$ denote a sequence of distributions verifying $T_s\rightharpoonup T_1$ in $\mathcal{D}'(\Omega)$, as $s\nearrow1$, and $supp (T_s)\subset F_s$, with $F_s\in\mathcal{H}(\bar{\omega})$,  for every $0<s<1$. Then if $F_s\to F_1$ in $\mathcal{H}(\bar{\omega})$, one has $supp (T_1)\subset F_1$.
\end{lemma}

\begin{remark}
    The strict interior convergence in the closure of the smooth open set $\omega \Subset\Omega$ is necessary because we only have a uniform H\"older regularity of the solutions $u^s$ as $s\nearrow1$ only in the interior of $\Omega$, as given in the second part of Theorem \ref{LocalRegEllipNonlin1}. It would be interesting to see if the global estimate of \cite{IannizzottoMosconiSquassina2016RMI_FracPLapGlobalHolder} can be made robust up to the boundary as $s\nearrow1$, 
    so that we can take $\omega$ up to the boundary of $\Omega$ instead of being in the interior.
\end{remark}

\begin{remark} It may be useful to recall the following known properties of the Hausdorff convergence of compact subsets of $\mathbb{R}^n$ with respect to other different types of convergence, from
Theorems 3:4.5 and 3:4.6 of \cite{ObstacleProblems}, and their references, for instance. 

    (i) The Hausdorff convergence $d_\mathscr{H} (C_s,C)\to 0$ implies the Kuratowski convergence  
    $C_s\xrightarrow[]{K}C$, which means that for all $x\in C$, there exists $x_s\in C_s$ such that $x_s\to x$, and for any subsequence $x_s\in C_s$, its limit $x\in C$.

    (ii) If $C$ and $\hat{C}$ are bounded convex subsets of $\mathbb{R}^n$, then 
    \[d_\mathscr{H}(C,\hat{C}) = d_\mathscr{H}(\partial C, \partial\hat{C}).\]

    (iii) For the class of uniformly Lipschitz bounded domains $C_s$, their convergence to $C$ in Lebesgue measure $d_\mathcal{L}(C_s,C):=\int_{\bar{\omega}}|\chi_{C_s}-\chi_C|\,dx\to0$, is equivalent to the convergence in Hausdorff distance $d_\mathscr{H}(C_s,C)\to0$, as well as to $d_\mathscr{H}(\partial C_s,\partial C)\to0$ and also, locally, to the equivalence of the uniform convergence of their graphs $\partial C_s\to\partial C$.
\end{remark}

\begin{remark}
The local topological assumption \eqref{ConvgH2} is in fact an assumption on the local regularity of the limit free boundary. For instance, in the case of $p=2$ for the Laplacian, if  $f\in W^{1,1}(\Omega)$, then it is known that the local coincidence sets are Caccioppoli sets \cite{Caffarelli1981pLapCoincidenceStability}, which means that the free boundary is $(n-1)$-rectifiable, as discussed in Theorem 6:8.6 of \cite{ObstacleProblems} in a slightly more general case. With that regularity, the reduced boundary, given by the points which have positive upper Lebesgue densities, may be written as a countable union of $C^1$ hypersurfaces. Additionally, with more regularity on $f$, much more on the singular points of the free boundary has been established (see \cite{Figalli++2020_IHES_GenericRegularityOfFreeBoundar}). For $p\neq 2$ in the obstacle problem with zero source $f$ and nonzero nondegenerate smooth obstacle, which is not exactly equivalent to the obstacle problem treated here, in Theorem 1.8 of \cite{Figalli++2017_JDE_RegFB_p-ObsPb} the authors proved that the free boundary is countably  $(n-1)$-rectifiable, which is the only regularity known in the nonlinear case.
\end{remark}

\begin{remark}
Indeed, it is not known how the topological assumption \eqref{ConvgH2} can be generalised to the fractional case for $0<s<1$. The only results in this direction available are for the fractional Laplacian, in \cite{BarriosFigalliRosoton2018AJMRegFreeBdryFracLap}. In this work, the authors showed that when the obstacle problem has 0 source function and sufficiently smooth obstacle, the free boundary can be decomposed into a union of regular points and singular points, and the singular points are locally contained within a stratified union of $C^1$ submanifolds, giving the same structure as the free boundary in the local case $s=1$.
\end{remark}

\appendix
\appendix

\noindent \textbf{Acknowledgements.} 
The research of C. Lo was partially done under the FCT PhD fellowship in the framework of the LisMath doctoral programme at the University of Lisbon, and was partially completed as a postdoc in the Liu Bie Ju Centre for Mathematical Sciences at the City University of Hong Kong. The research of J. F. Rodrigues was partially done under the framework of the Project UIDB/04561/2020 at CMAFcIO/ULisboa.

\printbibliography

\end{document}